\documentclass[11pt]{amsart}

\usepackage[a4paper,hmargin=3.5cm,vmargin=4cm]{geometry}
\usepackage{amsfonts,amssymb,amscd,amstext}
\usepackage{graphicx}
\usepackage[dvips]{epsfig} 
\usepackage{quoting}
\usepackage{hyperref} 

\usepackage{fancyhdr}
\input xy
\xyoption{all}

\usepackage{enumerate}
\usepackage{titlesec}
\usepackage{mathrsfs}

\titleformat{\section}
{\filcenter\bfseries} {\thesection{.}}{0.2cm}{}
\titleformat{\subsection}[runin]
{\bfseries} {\thesubsection{.}}{0.15cm}{}[.]
\titleformat{\subsubsection}[runin]
{\em}{\thesubsubsection{.}}{0.15cm}{}[.]

\usepackage[up,bf]{caption}

\newtheorem{theorem}{Theorem}[section]

\newtheorem{claim}[theorem]{Claim}
\newtheorem{corollary}[theorem]{Corollary}

\theoremstyle{definition}
\newtheorem{definition}[theorem]{Definition}
\newtheorem{remark}[theorem]{Remark}

\numberwithin{equation}{section}
\numberwithin{figure}{section}

\newcommand\Ccal{\mathcal{C}}
\newcommand\Dcal{\mathcal{D}}

\newcommand\Ical{\mathcal{I}}
\newcommand\Kcal{\mathcal{K}}
\newcommand\Lcal{\mathcal{L}}

\newcommand\Fcal{\mathcal{F}}

\newcommand\D{\overline{\mathbb D}}

\renewcommand\D{\mathbb D}

\newcommand\N{\mathbb{N}}

\newcommand\Q{\mathbb{Q}}
\newcommand\R{\mathbb{R}}

\newcommand\Z{\mathbb{Z}}

\renewcommand\c{\mathbb{C}}

\renewcommand\d{\mathbb D}

\newcommand\n{\mathbb{N}}
\renewcommand\r{\mathbb{R}}

\newcommand\z{\mathbb{Z}}

\newcommand\pgot{\mathfrak{p}}

\newcommand\ggot{\mathfrak{g}}

\newcommand\dist{\mathrm{dist}}

\newcommand\CMI{\mathrm{CMI}}

\def\dist{\mathrm{dist}}

\usepackage{color}

\begin{document}


\fancyhead[LO]{Generic properties of minimal surfaces }
\fancyhead[RE]{A.\ Alarc\'on and F.J.\ L\'opez}
\fancyhead[RO,LE]{\thepage}

\thispagestyle{empty}


\begin{center}
{\bf\Large Generic properties of minimal surfaces 
}

\medskip

%
%
{\bf Antonio Alarc\'on\,  and\,  Francisco J.\ L\'opez}
\end{center}

%
%
\medskip

\begin{quoting}[leftmargin={7mm}]
{\small
\noindent {\bf Abstract}\hspace*{0.1cm}
Let $M$ be an open Riemann surface and  $n\ge 3$ be an integer. In this paper we establish some generic properties 
 (in Baire category sense) in the space of all conformal minimal immersions $M\to\R^n$ endowed with the compact-open topology, pointing out that  a generic such immersion is chaotic in many ways.  For instance, we show that a generic conformal minimal immersion $u\colon M\to \r^n$ is non-proper, almost proper, and $\ggot$-complete with respect to any given Riemannian metric $\ggot$ in $\r^n$. Further, its image $u(M)$ is dense in $\r^n$ and disjoint from $\Q^3\times \r^{n-3}$, and has infinite area, infinite total curvature, and  unbounded curvature on every open set in $\r^n$. In case $n=3$, we also prove that a generic conformal minimal immersion $ M\to \r^3$  has infinite index of stability on every open set in $\r^3$.

\noindent{\bf Keywords}\hspace*{0.1cm} 
Minimal surface, Riemann surface, 
completely metrizable space, generic property, residual set, Baire category theorem.


\noindent{\bf Mathematics Subject Classification (2020)}\hspace*{0.1cm} 
53C42, 
53A05, 
54E52. 
}
\end{quoting}


\section{Introduction and main results}\label{sec:intro}

\noindent
A set in a topological space $X$ is {\em residual} (or {\em comeagre}) 
if it contains a countable intersection of dense open sets. 
If $X$ is completely metrizable (or, more generally, a Baire space) then every residual set in $X$ is dense.\footnote{This property characterizes Baire spaces. Every completely metrizable space is a Baire space by the Baire Category Theorem; see, e.g., \cite[Corollary 25.4, p.\ 186]{Willard1970}.}
%
%
A countable intersection of residual sets is still residual, so residual sets in completely metrizable spaces are considerably {\em large}. We say that a property of elements in a completely metrizable space is {\em generic}, or that a {\em generic} element satisfies the property, if the property holds on a residual subset (see, e.g., \cite[Definition 8.5]{Kechris1995}). Thus, generic properties are those enjoyed by {\em almost all} elements of the space in this precise topological sense.

Throughout the paper, $n\ge 3$ is an integer, $M$ is an open Riemann surface, and ${\rm CMI}(M,\R^n)$ denotes the space of all conformal minimal immersions $M\to\R^n$ endowed with the compact-open topology. A sequence $u_j\in {\rm CMI}(M,\R^n)$, $j\in\N$, converges in this topology to an immersion $u\in {\rm CMI}(M,\R^n)$ if and only if the sequence of restrictions $u_j|_K$ to any compact set $K\subset M$ converges uniformly on $K$ to $u|_K$ (see, e.g., \cite[\textsection VII.2]{Bredon1993}).  
In this paper we shall establish several generic properties in ${\rm CMI}(M,\R^n)$. Some of them are somewhat surprising, and show that a generic conformal minimal immersion $M\to\R^n$ has a considerably wild asymptotic behavior. 
The following are some of the properties we are interested in.
We denote by $|\cdot|$ and $\dist(\cdot,\cdot)$ the standard Euclidean norm and distance in $\r^n$, respectively. 
%
%
\begin{definition}\label{de:def}
(a)  Given a Riemannian metric $\ggot$ in $\R^n$, an immersion $u\in {\rm CMI}(M,\R^n)$ is $\ggot$-{\em complete} if the pull-back metric $u^*\ggot$ induced on $M$ by $\ggot$ via $u$ is complete in the classical sense:   $u\circ\gamma$ has infinite $\ggot$-length for every divergent path $\gamma:[0,1)\to M$, that is, $\int_0^1|(u\circ\gamma)'(t)|_\ggot dt=+\infty$. 
The immersion $u$ is {\em complete} if it is $ds^2$-complete, where $ds^2$ denotes the standard Euclidean metric in $\R^n$.

(b) A map between topological spaces is {\em proper} if inverse images of compact subsets are compact. It is {\em almost proper} if the connected components of inverse images of compact subsets are compact.

(c) Denote by $\overline\D=\{z\in\c\colon |z|\le1\}$ the closed unit disc. For $u\in\CMI(M,\R^n)$ and a conformal minimal immersion $v:\overline \D\to\R^n$, we say that $u$ {\em rebuilds} $v$ if for any $\epsilon>0$ 
 there are a smoothly bounded closed disc $D$ in $M$
  and a biholomorphism $\varphi:\overline\D\to D$ such that $|u\circ\varphi-v|<\epsilon$ everywhere on $\overline\D$. We shall say that $u$ {\em rebuilds every conformal minimal disc} if $u$ rebuilds every conformal minimal immersion $\overline \D\to\R^n$ (compare with the notion of {\em universal} minimal surface in \cite[Definition 4.1]{Lopez2014JGA}).

(d) For $u\in {\rm CMI}(M,\R^n)$, we denote
\begin{equation}\label{eq:DP}
	{\rm DP}(u) 
	=\big\{ p\in M\colon u^{-1}(u(p))\setminus\{p\}\neq\varnothing\big\}.
\end{equation}
The immersion $u$ is injective if and only if ${\rm DP}(u)=\varnothing$. We say that $u$ is {\em densely non-injective} if ${\rm DP}(u)$ is a dense subset of $M$.

(e)  Denote by $D_M=\{(p,p)\colon p\in M\}\subset M\times M$ the diagonal of $M\times M$. For $u\in {\rm CMI}(M,\R^n)$, let $\sigma_u\colon M\times M\to \r^n$ be the difference map given by 
\[
	\sigma_u(p,q)=u(p)-u(q),\quad p,q\in M. 
\]
We have that $D_M\subset \sigma_u^{-1}(0)$, and $u$ is injective if and only if $D_M=\sigma_u^{-1}(0)$. We say that $u$ {\em self-intersects nicely} if $0\in\R^n$ is a regular value of the restriction of $\sigma_u$ to $M\times M\setminus D_M$; i.e., $d(\sigma_u)_{(p,q)}:T_{(p,q)}M\times M\to \R^n$ is surjective for all $(p,q)\in \sigma_u^{-1}(0)\setminus D_M$.
\end{definition}

Recall that a subset of a topological space is a $G_\delta$ if it is a countable intersection of open sets; so, dense $G_\delta$ subsets are residual.  For $u\in {\rm CMI}(M,\R^n)$ we denote by $K_u\colon M\to \r$ and $dA_u$  the curvature function and area element  of $u$, respectively. Non-constant continuous   functions $\r^n\to [0,+\infty[$ are called {\em weight functions} on $\r^n$. 

Here is the first result in this paper. 
%
%
\begin{theorem}\label{th:open}
Let $M$ be an open Riemann surface and $n\ge 3$ be an integer. Then the space ${\rm CMI}(M,\R^n)$ is completely metrizable and separable. 

Moreover, given a  Riemannian metric $\ggot$ in $\R^n$, a  conformal minimal disc $v:\overline \d\to \r^n$, a closed   set  $\varnothing\neq Z\subset \r^n$,  a weight function $\rho$ on $\r^n$,   a properly immersed submanifold  $A\subset \R^n$ of codimension $\ge 2$, and a  compact subset $S\subset A$  with empty (relative) interior,   the following subspaces are dense $G_\delta$ subsets of  ${\rm CMI}(M,\R^n)$.
\begin{enumerate}[(i)]
\item The set of $\ggot$-complete immersions.
\item  The set of immersions rebuilding $v$.
\item  The set of immersions $u$  with $\dist(u(M),Z)=0$.
\item The set of immersions $u$ such that  $\sup_M |K_u| (\rho\circ u)=+\infty$.
\item The set of immersions $u$  such that $\int_M |K_u| (\rho\circ u) dA_u=+\infty$.
\item  The set of immersions $u$  such that $\int_M  (\rho\circ u) dA_u=+\infty$.
\item The set of immersions with range in $\R^n\setminus S$.  
\item The set of immersions that self-intersect nicely.
\end{enumerate}
Furthermore, 
\begin{enumerate}[(i)]
\item[(ix)] the set of almost proper immersions is residual in ${\rm CMI}(M,\R^n)$. 
\end{enumerate}
\end{theorem}
In particular, a generic conformal minimal immersion in  ${\rm CMI}(M,\R^n)$ satisfies all the properties in Theorem \ref{th:open}.
Recall that a subspace of a completely metrizable space is completely metrizable if and only if it is a $G_\delta$; see \cite[Theorem 24.12, p.\ 179]{Willard1970}. Therefore, the subspaces in statements  {\em (i)}--{\em (viii)} in Theorem \ref{th:open} are completely metrizable. 
Observe that if a Riemannian metric $\ggot$ in $\R^n$ is complete then every almost proper immersion $M\to\R^n$ is $\ggot$-complete, but this need not hold true if $\ggot$ is not complete. In view of statement {\em (ix)} in the theorem, the main interest of {\em (i)} is when the given metric $\ggot$ in $\R^n$ is not complete (for instance, if it decays fast at infinity). Moreover, we do not know whether the set of almost proper immersions in ${\rm CMI}(M,\R^n)$ is a $G_\delta$ subset. 

Satisfying a given countable collection of generic properties in a completely metrizable space is again a generic property. Therefore, the following corollary of Theorem \ref{th:open}, which might be understood as the main result in this paper,  holds.
%
%
\begin{corollary}\label{co:open}
If $M$ is an open Riemann surface and $n\ge 3$  is an integer, then the following subspaces of ${\rm CMI}(M,\R^n)$ are dense $G_\delta$ subsets.
\begin{enumerate}[(i)]
\item The set of immersions   which are $\ggot_j$-complete for all $j\in \n$, where $\ggot_j$, $j\in\N$, is any given sequence of Riemannian metrics in $\R^n$.
\item The set of immersions rebuilding every conformal minimal disc.
\item The set of immersions   with dense image.
\item  The set of immersions $u\in {\rm CMI}(M,\R^n)$ having dense image and  unbounded curvature on $u^{-1}(\Omega)$ for every open set $\Omega\subset \r^n$.
\item The set of immersions $u\in {\rm CMI}(M,\R^n)$  having dense image and infinite total curvature on $u^{-1}(\Omega)$ for every open set $\Omega\subset \r^n$.
\item  The set of immersions $u\in {\rm CMI}(M,\R^n)$ having dense image and infinite area on $u^{-1}(\Omega)$ for every open set $\Omega\subset \r^n$.
\item The set of immersions with range in $\R^n\setminus \bigcup_{j\in \n}S_j$ for any given sequence  of   properly immersed submanifolds $A_j\subset\R^n$ of codimension $\ge 2$ and meagre subsets $S_j\subset A_j$, $j\in \n$.  
\end{enumerate}
In particular, the intersection of all these subspaces is still a dense $G_\delta$ subset.  
\end{corollary}
The given metrics $\ggot_j$ in Corollary \ref{co:open}{\em (i)}  need not be complete, or conformal or comparable to each other or to the standard Euclidean metric. Roughly speaking, {\em (i)} shows that divergent paths in a generic minimal surface have an extremely wild behavior, and hence such surfaces are {\em complete} in a very strong sense.
 In order to obtain {\em (ii)} we shall use   that the space  $\CMI(\overline\d,\R^n)$ ($n\geq 3$) of all conformal minimal immersions $\overline\D\to\R^n$, endowed with the compact-open topology, is separable; see Claim \ref{cl:disc-sep}. 
 %
%
%
%
%
%
For the proof of statements  {\em (iii)}--{\em (vi)} in Corollary \ref{co:open} we take into account that $\r^n$ is separable and second-countable.
 Concerning {\em (vii)}, recall that a set in a topological space $X$ is {\em meagre} (or of {\em first category}) if its complement is residual in $X$; i.e., if the set is contained in a countable union of closed sets with empty interior. The family of all meagre sets in $X$ is a {\em $\sigma$-ideal}, that is, the empty set is meagre in $X$ and 
all subsets of a meagre set and all countable unions of meagre sets in $X$ are still meagre.
A  consequence of Corollary \ref{co:open}{\em (vii)} is that,  for any given sequence  of   properly immersed submanifolds $A'_j\subset\R^n$, $j\in \n$, of codimension $\ge 3$,  the set of immersions in $ {\rm CMI}(M,\R^n)$  with range in $\R^n\setminus \bigcup_{j\in \n}A'_j$ is a dense $G_\delta$. 
 In particular,  $u(M)\cap(\Q^3\times \r^{n-3})=\varnothing$ holds for a generic conformal minimal immersion $u:M\to\R^n$ (despite $u(M)$ being generically dense in $\r^n$ by Corollary \ref{co:open}{\em (iii)}).
 
In the same spirit as Corollary \ref{co:open}{\em (iv)--(vi)}, we shall also establish the following result in dimension $n=3$.
 \begin{theorem}\label{th:index}
 If $M$ is an open Riemann surface, then the set of immersions $u\in {\rm CMI}(M,\R^3)$ having dense image and infinite index of stability on $u^{-1}(\Omega)$ for every open set $\Omega\subset \r^3$  is a dense $G_\delta$ subset.
 \end{theorem}
 Recall that the {\em index of stability} of an immersion  $u\in {\rm CMI}(M,\R^3)$  is the index of its Jacobi operator $\Delta_u-2K_u$, where $\Delta_u$ is the Laplace operator of $u^*(ds^2)$; see, e.g., \cite[Sec.\ 2.8]{MeeksPerez2012} for a brief introduction to stability of minimal surfaces.

The set of proper conformal minimal immersions in ${\rm CMI}(M,\R^n)$ is  known to be dense (see \cite[Theorem 7.1]{AlarconForstnericLopez2016MZ} or \cite[Theorem 3.10.3]{AlarconForstnericLopez2021Book}), but it fails to be residual. Indeed, since a conformal minimal immersion $M\to\R^n$ with dense image is not proper, the following is an immediate consequence of Corollary \ref{co:open}{\em (iii)}.
%
%
\begin{corollary}\label{co:proper}
Let $M$ be an open Riemann surface and $n\ge 3$ be an integer. Then the set of proper conformal minimal immersions is meagre in ${\rm CMI}(M,\R^n)$. 
\end{corollary}
Recall that meagre sets in a completely metrizable space have empty interior.
So, in a certain topological sense, the meagre sets in such a space can be considered {\em small} and even {\em negligible}. In particular, a set in a completely metrizable space cannot be both meagre and residual.

The following result is a consequence of Theorem \ref{th:open}{\em (ii)} and {\em (viii)}, and the fact that conformal minimal immersions $M\to \r^n$, $n=3,4$, rebuilding every conformal minimal disc are densely non-injective (see Definition \ref{de:def} {\rm (d)} and Claim \ref{cl:DP}).  Recall that an immersion $u\in {\rm CMI}(M,\R^4)$ has {\em simple double points} if for any $p,q\in M$, $p\neq q$, with $u(p)=u(q)$, the tangent planes $du_p(T_pM)$ and $du_q(T_qM)$ intersect only at $0\in \r^4$, and there are no triple intersections.
%
%
\begin{corollary}\label{co:DP}
Let $M$ be an open Riemann surface. Then the following hold.
\begin{enumerate}[(i)]
\item If $n\ge 5$ then the set of injective immersions is a dense $G_\delta$ set in ${\rm CMI}(M,\R^n)$.
\item The set of immersions $u\in {\rm CMI}(M,\R^4)$ with simple double points is a dense $G_\delta$ subset, and the set of those $u$ such that, in addition, ${\rm DP}(u)$ is a countable dense set in $M$ is residual ${\rm CMI}(M,\R^4)$.
\item A generic immersion $u\in {\rm CMI}(M,\R^3)$   satisfies that ${\rm DP}(u)$ is a dense countable union of properly immersed real curves  in $M$.  
\end{enumerate}
\end{corollary}

The subspaces of ${\rm CMI}(M,\R^n)$ in Corollary \ref{co:open}{\em (iii)} and Theorem \ref{th:open}{\em (ix)} are known to be dense; see \cite[Theorem 1.1]{AlarconCastro-Infantes2018GT} and \cite[Theorem 7.1]{AlarconForstnericLopez2016MZ} (see also \cite[Theorem 3.9.1 and 3.10.3]{AlarconForstnericLopez2021Book}). The subspace in Theorem \ref{th:open}{\em (i)} is also known to be dense in ${\rm CMI}(M,\R^n)$ when $\ggot$ is the standard Euclidean metric $ds^2$ in $\R^n$ (see  \cite[Theorem 3.9.1 {\em i)}]{AlarconForstnericLopez2021Book}); it is not difficult to modify the proof in the mentioned source to make it work for an arbitrary metric $\ggot$ in $\R^n$, though. The proofs of these density results follow from some intricate lemmas that are recursively applied in involved inductive procedures. Our method of proof of Theorem \ref{th:open} is much simpler, relying on a single application of the same lemmas together with the Baire category theorem, so without requiring of any induction, and gives the stronger result that the concerned subspaces are not only dense but residual in ${\rm CMI}(M,\R^n)$. The fact that the subspaces in Corollary \ref{co:open}, apart from the one in  {\em (iii)}, and Theorem \ref{th:index} are dense in ${\rm CMI}(M,\R^n)$  is new; in fact, the referred subspaces were not known to be nonempty until now.

In conclusion, it is safe to claim that the shape of {\em almost all} minimal surfaces in Euclidean space is extremely complicated.


\section{Proof of Theorem \ref{th:open}}\label{sec:genericCMI}

\noindent
Throughout the proof we let 
\begin{equation}\label{eq:exhaustion}
	\varnothing\neq K_1\Subset K_2\Subset\cdots\subset \bigcup_{j\in\n}K_j=M
\end{equation}
be an exhaustion of $M$ by smoothly bounded Runge compact domains.\footnote{A compact set $K$ in an open Riemann surface $M$ is {\em Runge} (or {\em holomorphically convex}) if  $M\setminus K$ has no relatively compact connected components in $M$. It is customary to call a nonempty set in a topological space a {\em compact domain} if it is compact and is the closure of a connected open set. Lacking a better term, in this paper we also refer to the union of finitely many mutually disjoint compact domains as a compact domain.}

Let us first prove the first assertion in the theorem.
%
%
\begin{claim}\label{cl:metrizable}
Let $M$ be an open Riemann surface and $n\ge 3$ be an integer. Then the space ${\rm CMI}(M,\R^n)$ is completely metrizable.
\end{claim}
\begin{proof} 
Denote by $\Ccal(M,\R^n)$ the space of all continuous maps $M\to\R^n$ endowed with the compact-open topology.  
Since $M$ is hemicompact\footnote{A topological space $X$ is {\em hemicompact} if there is a countable family of compact sets in $X$ such that every compact set in $X$ is contained in one of those in the family.} 
and $\R^n$ is complete with the Euclidean metric,  
we have that $\Ccal(M,\R^n)$ is completely metrizable (see Arens \cite{Arens1946}; see also, e.g., 
 \cite[p.\ 100]{McCoyNtantu1988}). 
Recall that a map $u=(u_1,\ldots,u_n):M\to\R^n$ lies in ${\rm CMI}(M,\R^n)$ if and only if $u$ is a harmonic map and its complex derivative $\partial u$ (i.e., the $(1,0)$-part of the exterior derivative ${\rm d}u$ of $u$) satisfies $\sum_{i=1}^n (\partial u_i)^2=0$ and $\partial u\neq 0$ everywhere on $M$ (see, e.g.,  \cite[Theorem 2.3.1]{AlarconForstnericLopez2021Book}).
Denote by ${\rm CHM}(M,\R^n)$ the subspace of $\Ccal(M,\R^n)$ consisting of all 
harmonic maps $u:M\to\R^n$ satisfying $\sum_{i=1}^n (\partial u_i)^2=0$ everywhere on $M$. By Harnack's theorem ${\rm CHM}(M,\R^n)$ is a closed subspace of $\Ccal(M,\R^n)$, hence ${\rm CHM}(M,\R^n)$ is completely metrizable as well. Recall now that a subspace of a completely metrizable space is completely metrizable if and only if it is a $G_\delta$ set (see, e.g., \cite[Theorem 24.12, p.\ 179]{Willard1970} or \cite[Theorems 4.3.23 and 4.3.24, p.\ 274]{Engelking1989}). Since 
\[
	{\rm CMI}(M,\R^n)=\big\{u\in{\rm CHM}(M,\R^n)\colon \partial u(p)\neq 0  \text{ for all } p\in M\big\}
\]
is a subspace of ${\rm CHM}(M,\R^n)$, to complete the proof it remains to see that ${\rm CMI}(M,\R^n)$ is a $G_\delta$ set in ${\rm CHM}(M,\R^n)$; i.e., it can be written as a countable intersection of open sets. For that, it is then clear that ${\rm CMI}(M,\R^n)=\bigcap_{j\in\N} U_j$,
where
\[
	U_j=\big\{u\in {\rm CHM}(M,\R^n)\colon\partial u(p)\neq 0  \text{ for all } p\in K_j\big\},\quad j\in\N.
\]
(See \eqref{eq:exhaustion}.)
To finish, just note that $U_j$ is an open subset in $ {\rm CHM}(M,\R^n)$ for all  $j\in\n$ by Cauchy estimates. 
\end{proof}
\begin{claim}\label{cl:disc-sep}
Let $M$ be an open Riemann surface, $K$ be a smoothly bounded compact domain in $M$, and $n\geq 3$ an integer. Then the space   $\CMI(K,\R^n)$ of all conformal minimal immersions $K\to \r^n$ (i.e., extending as conformal minimal immersion to an unspecified open neighborhood of $K$),  endowed with the compact-open topology,   is separable.   The same holds for    $\CMI(M,\R^n)$.
\end{claim}
\begin{proof}
Let $\Ccal(K,\r^n)$ denote the space of all continuous maps $K\to\R^n$ endowed with the compact-open topology. Since $K$ is compact Hausdorff and $\R^n$ is a metric space, this topology coincides with the metric topology of the maximum norm (see, e.g., \cite[Theorem 2.12, p.\ 440]{Bredon1993}). Note that $\Ccal(K,\r^n)$ is separable by Riesz's Theorem; see, e.g., \cite[p.\ 251]{RoydenFitzpatrick2010}, and take into account that the product of finitely many separable metric spaces is separable.
Since every subspace of a separable metric space is separable (see, e.g., \cite[Proposition 26, p.\ 204]{RoydenFitzpatrick2010} or  \cite[16G.1]{Willard1970}),  we have that   $\CMI(K,\R^n)\subset \Ccal(K,\r^n)$   is separable too.  
 
 For the second assertion, note that $\CMI(K_j,\R^n)$ is separable for all $j\in \n$ (see \eqref{eq:exhaustion}). Since the Runge theorem for conformal minimal immersions in  \cite[Theorem 3.6.1]{AlarconForstnericLopez2021Book} (see Theorem \ref{th:Runge} below) enables us to uniformly approximate every element in $\CMI(K_j,\R^n)$ by a sequence in $\CMI(M,\R^n)$, the latter space is separable as well.
 \end{proof}

We shall next  focus on statements {\em (i)}--{\em (vi)}. Their proofs follow a common logical pattern using the following immediate consequence of the Baire Category Theorem. For a compact topological space $K$ and  a continuous map $f\colon K\to \r^n$ we denote by $\|f\|_K=\max\{|f(p)|\colon p\in K\}$.
\begin{claim}\label{cl:patron}
Let $M$ be an open Riemann surface and $n\geq 3$ be an integer. 
Let $E$ be a family of compact subsets of $M$  and  
$\Fcal\colon  {\rm CMI}(M,\R^n) \times E\to [0,+\infty]$ 
  be a function satisfying the following conditions:
  \begin{enumerate}[{\rm (A)}]
  \item   $\Fcal(\cdot,C)\colon   {\rm CMI}(M,\R^n) \to [0,+\infty]$ is continuous for all $C\in E$, where $ [0,+\infty]$ has the topology of the extended real line.
  \item  For every $i\in \n$  the set 
\begin{equation}\label{eq:Lambda-i}
\Lambda_i=\{u\in   {\rm CMI}(M,\R^n) \colon \sup_E  \Fcal(u,\cdot)>i\}
\end{equation}
is dense in $  {\rm CMI}(M,\R^n)$, that is, for any $u\in   {\rm CMI}(M,\R^n)$, any compact set $K\subset M$, and any $\epsilon>0$, there exists $\tilde u\in \Lambda_i$ with $\|\tilde u-u\|_K<\epsilon$.
  \end{enumerate}
  Then the set $\bigcap _{i\in \n} \Lambda_i$ is a dense $G_\delta$ subset in $  {\rm CMI}(M,\R^n)$. 
  \end{claim}
\begin{proof}  By the continuity of $\Fcal(\cdot,C)$, the set $\Lambda_i^C=\{u\in   {\rm CMI}(M,\R^n) \colon  \Fcal(u,C)>i\}$ is open in $  {\rm CMI}(M,\R^n) $ for all $C\in E$ and $i\in \n$, and so is $\Lambda_i=\bigcup_{C\in E} \Lambda_i^C$ for all $i\in \n$. Since $ {\rm CMI}(M,\R^n)$ is  completely metrizable by  Claim \ref{cl:metrizable} and $\Lambda_i$ is also dense in $  {\rm CMI}(M,\R^n)$, the Baire Category Theorem ensures that    $\bigcap _{i\in \n} \Lambda_i$ is a dense $G_\delta$ subset in ${\rm CMI}(M,\R^n)$.  
\end{proof}

In order to establish each of the statements, we shall consider a pair $(E,\Fcal)$ as in   Claim \ref{cl:patron} such that $\bigcap _{i\in \n} \Lambda_i$ equals the set we are dealing with. Condition {\rm (A)} in  the claim will always be clear, either trivially or from a straightforward application of the Cauchy estimates, and we will not discuss it (except when proving statement {\em (i)}, serving to exemplify how the Cauchy estimates are used). Thus,  the proof reduces to check condition {\rm (B)}. For this, we shall use the following Runge theorem for conformal minimal immersions (see \cite[Theorem 5.3]{AlarconForstnericLopez2016MZ} or \cite[Theorem 3.6.1]{AlarconForstnericLopez2021Book}), together with further results in the recently developed theory of approximation for minimal surfaces.
%
%
\begin{theorem}\label{th:Runge}
Let $M$ be an open Riemann surface, $K\subset M$ be a Runge compact set, and $u:K\to\R^n$ $(n\ge 3)$ be a conformal minimal immersion (on a neighborhood of $K$ in $M$). Then, for any $\epsilon>0$ there is $\hat u\in {\rm CMI}(M,\R^n)$ such that $|\hat u-u|<\epsilon$ everywhere on $K$.
\end{theorem}
Let us now proceed to prove items {\em (i)} to {\em (ix)} in Theorem \ref{th:open}. This shall complete the proof of the theorem.
\begin{proof}[Proof of (i)] Let $\ggot$ be a Riemannian metric in $\R^n$. Fix a point $p_0\in\mathring K_1$, consider the family of compact sets $E=\{bK_j\colon j\in \n\}$  (see \eqref{eq:exhaustion}),  and set 
\[
\Fcal(u,C)= \dist_u^\ggot(p_0,C) \text{ for every $u\in {\rm CMI}(M,\R^n)$ and $C\in E$,}
\]
where $\dist_u^\ggot$ denotes the distance on $M$ associated to the pull-back metric $u^*\ggot$. Let $C=b K_j\in E$, $u\in  {\rm CMI}(M,\R^n)$, and $\mu>0$. The Cauchy estimates (relying on the fact that the uniform convergence on a compact set of harmonic or holomorphic functions implies the one of all their derivatives) guarantee the existence of a number $\sigma>0$ with the following property: for any $v\in  {\rm CMI}(M,\R^n)$ with $\|v-u\|_{K_j}<\sigma$ the metric $v^*\ggot$ is so close to $u^*\ggot$ on $K_j$ that  $|\dist_v^\ggot(p_0,C)- \dist_u^\ggot(p_0,C)|<\mu$. This proves  Claim \ref{cl:patron} {\rm (A)}.

Let $i\in \n$ and fix $u$, $K$, and $\epsilon$ as in   Claim \ref{cl:patron} {\rm (B)}. Choose $C=bK_j\in E$,  where $j\in\N$ is so large that $K\subset\mathring K_j$, and fix a number $\delta>0$ to be specified later. By \cite[Lemma 4.1]{AlarconDrinovecForstnericLopez2015PLMS} (see also \cite[Lemma 7.3.1]{AlarconForstnericLopez2021Book}) there is a conformal minimal immersion $\hat u:K_j\to\R^n$ such that $\|\hat u-u\|_{K_j}<\epsilon/2$ and $\dist_{\hat u}(p_0, C)>\delta$, where $\dist_{\hat u}$ denotes the distance function on $M$ associated to the metric $\hat u^*{ds^2}$ induced on $M$ via $\hat u$ by the standard Euclidean metric $ds^2$ in $\R^n$. Since $ds^2$ and the given metric $\ggot$ are comparable in the compact set $\{x\in\R^n\colon \dist(x,u(K_j))\le \epsilon\}\subset\R^n$, we can choose  $\delta>0$ so large that 
\begin{equation}\label{eq:i+1}
	\dist_{\hat u}^\ggot(p_0,C)>i.
\end{equation}
Next, by Theorem \ref{th:Runge} we can approximate $\hat u$ uniformly on $K_j$ by immersions $\tilde u\in\CMI(M,\R^n)$ with $\|\tilde u-\hat u\|_{K_j}<\epsilon/2$, and hence $\|\tilde u-u\|_K<\epsilon$. Moreover, if the approximation of $\hat u$ by $\tilde u$ on $K_j$ is close enough then, by \eqref{eq:i+1} and Cauchy estimates, we can also ensure that $\dist_{\tilde u}^\ggot(p_0,C)>i$, and hence $\tilde u\in \Lambda_i$; see \eqref{eq:Lambda-i}.  This shows {\rm (B)} and completes the proof in view of  Claim \ref{cl:patron}; note that  $\bigcap _{i\in \n} \Lambda_i$ equals the set of $\ggot$-complete immersions in ${\rm CMI}(M,\R^n)$.
\end{proof}

\begin{proof}[Proof of (ii)]
Let $v:\overline \d\to \r^n$ be a conformal minimal immersion (on a neighborhood of $\overline \d$ in $\c$), consider the family $E=\{D\subset M\colon \text{ $D$ smoothly bounded closed disc}\}$, and for each $D\in E$ denote by $\Lcal_D$ the set of all biholomorphisms $\overline \d\to D$.  Set 
\[
\Fcal(u,D)=\sup_{\varphi\in \Lcal_D} \frac1{\|u\circ \varphi-v\|_{\overline \d}}  \in (0,+\infty] \quad \text{for $u\in {\rm CMI}(M,\R^n)$ and $D\in E$.}
\]
 Let $i\in \n$ and fix $u$, $K$, and $\epsilon$ as in   Claim \ref{cl:patron} {\rm (B)}. We assume without loss of generality  that  $K\subset M$ is Runge, and pick $D\in E$ with $K\cap D=\varnothing$ and $\varphi\in \Lcal_D$. Since $K\cup D$ is Runge, by Theorem \ref{th:Runge} there is $\tilde u\in  {\rm CMI}(M,\R^n)$ such that $|\tilde u-u|<\epsilon$ on $K$ and $|\tilde u-v\circ \varphi^{-1}|<1/i$ on $D$, hence $\tilde u\in \Lambda_i$ (see \eqref{eq:Lambda-i}), proving {\rm (B)}. Thus,  Claim \ref{cl:patron}  ensures that the set
$\bigcap _{i\in \n} \Lambda_i$, which equals the one of conformal minimal immersions $M\to \r^n$  rebuilding $v$,
is a dense $G_\delta$ subset in $ {\rm CMI}(M,\R^n)$.
\end{proof}
\begin{proof}[Proof of (iii)--(vi)] 
Let $\varnothing \neq Z\subset \r^n$ be closed set  and $\rho\colon \r^n\to \r$ be a weight function.    Let $E=\{K_j\colon j\in \n\}$  (see \eqref{eq:exhaustion}), and
%
 for each $3\leq l\leq 6$ define the function $\Fcal_l\colon   {\rm CMI}(M,\R^n)\times E\to [0,+\infty]$ given by
 \[
 \Fcal_l(u,C)=\left\{
 \begin{array}{ll}
\displaystyle  \frac1{\dist(u(C),Z)}
 & \text{if }l=3\smallskip
\\
\displaystyle\sup_{C} |\Kcal_u| (\rho\circ u)
 & \text{if }l=4\smallskip
\\
\displaystyle \int_{C} |\Kcal_u| (\rho\circ u)dA_u
 & \text{if }l=5\smallskip
\\
\displaystyle \int_{C} (\rho\circ u) dA_u
 & \text{if }l=6\;.
 \end{array}\right.
 \]
 Let $i\in \n$ and fix $u$, $K$, and $\epsilon$ as in   Claim \ref{cl:patron} {\rm (B)}. We assume without loss of generality  that  $K\subset M$ is Runge.   Let $j\in\N$ be so large that $K\subset K_j\in E$ and  take a smoothly bounded closed disc $D\subset K_{j+1}\setminus K_j$.  Let $B\subset \r^n$ be an open ball and $\delta>0$ such that  $\rho|_{\overline B}>\delta$.   By Theorem \ref{th:Runge} applied to a suitable extension of $u|_{K_j}$ to $K_j\cup D$, for each $l\in \{3,4,5,6\}$ there is $\tilde u^l\in    {\rm CMI}(M,\R^n)$ such that $|\tilde u^l-u|<\epsilon$ on $K$,  $\tilde u^l(D)\subset B$ provided that $4\leq l\leq 6$, and
\[
 \dist(\tilde u^3(D),Z)<1/i,\quad \sup_D |\Kcal(\tilde u^4)| >i/\delta,
 \]
 \[  \int_{D} |\Kcal(\tilde u^5)|dA_{\tilde u^5}>i/\delta, \quad {\rm Area}(\tilde u^6(D))=\int_{D}dA_{\tilde u^6}>i/\delta.
 \]
 For cases $l=5,6$, use that the open ball $B$ contains complete bounded minimal discs \cite{Nadirashvili1996IM}, having infinite total curvature and area.
 Since $\rho|_{\overline B}>\delta$ and $\tilde u^l(D)\subset B$, $4\leq l\leq 6$, this shows that $\tilde u^l\in\Lambda_i^l=\{u\in  {\rm CMI}(M,\R^n)\colon \sup_E  \Fcal_l(u,\cdot)>i\}$ and hence $\Lambda_i^l$ is dense in $ {\rm CMI}(M,\R^n)$ for all $i\in \n$ and $3\leq l \leq 6$. Thus, $\bigcap_{i\in \n} \Lambda_i^l$ is a dense $G_\delta$ in $ {\rm CMI}(M,\R^n)$ for each $l$ by Claim \ref{cl:patron}. Since these are the sets in  {\em (iii)--(vi)}, the proof is done.
\end{proof}

\begin{proof}[Proof of (vii)] 
Let $A\subset\R^n$ be a properly immersed submanifold of codimension $\ge 2$ and $S\subset A$ be a compact subset with empty interior. 
By Claim \ref{cl:metrizable}  and the Baire category theorem, it suffices to prove that  the set
\[
	\Theta_i=\Big\{u\in{\rm CMI}(M,\R^n)\colon
	u(K_i)\cap S=\varnothing\Big\}
\]
 (see \eqref{eq:exhaustion}) is open and dense in  ${\rm CMI}(M,\R^n)$  for every $i\in \n$; observe that 
 $\Theta=\bigcap_{i\in\N}\Theta_i$ 
 equals the set of immersions in ${\rm CMI}(M,\R^n)$ with range in $\R^n\setminus S$. Indeed, the openness is clear by compactness of $K_i$ and $S$.
For the density, fix $i\in\N$ and choose an immersion $u\in\CMI(M,\R^n)$, a compact set $K\subset M$, and a number $\epsilon>0$. Since $\Theta_i\supset \Theta_j$ for all $j\geq i$, we can assume that $i\in\N$ is so large that $K\subset K_i$. An inspection of the proof of the general position theorem in \cite[Theorem 3.4.1]{AlarconForstnericLopez2021Book} shows that there is a conformal minimal immersion $\hat u:K_i\to\R^n$ such that $\|\hat u-u\|_{K_i}<\epsilon/3$, $\hat u$ intersects $A\subset\R^n$ transversely, and $\hat u(bK_i)\cap A=\varnothing$; see \cite[p.\ 145]{AlarconForstnericLopez2021Book} and recall that $\dim A\le n-2$. In particular, the set $\hat u(K_i)\cap A$ is  finite, and hence so is $\hat u(K_i)\cap S$. Since $S\subset A$ is compact and has empty relative interior, post-composing $\hat u$ with a suitable small translation in $\R^n$ we can obtain a conformal minimal immersion $\hat u':K_i\to\R^n$ such that $\|\hat u'-\hat u\|_{K_i}<\epsilon/3$ and $\hat u'(K_i)\cap S=\varnothing$. Finally, Theorem \ref{th:Runge}  allows to approximate $\hat u'$ uniformly on $K_i$ by an immersion $\tilde u\in\CMI(M,\R^n)$ with $\|\tilde u-\hat u'\|_{K_i}<\epsilon/3$, and hence $\|\tilde u-u\|_{K_i}<\epsilon$. Taking into account that both $K_i$ and $S$ are compact, if the approximation of $\hat u'$ by $\tilde u$ on $K_i$ is close enough then $\tilde u(K_i)\cap S=\varnothing$, and hence $\tilde u\in \Theta_i$. This completes the proof.
\end{proof}
\begin{proof}[Proof of (viii)]  For each $i\in \n$ choose an open neighborhood $\Delta_i$ of the diagonal $D_M$ of $M\times M$ such that $\Delta_i\supset \Delta_{i+1}$, $i\in \n$, and $\bigcap_{i\in \n} \Delta_i=D_M$. For every $i\in \n$ denote by $\Theta_i$ the set of those immersions 
$u\in{\rm CMI}(M,\R^n)$ such that  $d(\sigma_u)_{(p,q)}:T_{(p,q)}M\times M\to \R^n$ is surjective for all $(p,q)\in \sigma_u^{-1}(0)\cap (K_i\times K_i)\setminus \Delta_i$; see Definition \ref{de:def} {\rm (e)} and  \eqref{eq:exhaustion}.
By Claim \ref{cl:metrizable}  and the Baire category theorem, it suffices to prove that  the set $\Theta_i$ is 
  open and dense in  ${\rm CMI}(M,\R^n)$  for every $i\in \n$; observe that 
 $\Theta=\bigcap_{i\in\N}\Theta_i$ 
 equals    the set of immersions in ${\rm CMI}(M,\R^n)$   that self-intersect nicely. To check that $\Theta_i$ is open, just observe that for a given $u\in \Theta_i$  the Cauchy estimates give a number $\epsilon>0$ such that 
 $d(\sigma_{\hat u})_{(p,q)}$ is surjective for all  $(p,q)\in \sigma_{\hat u}^{-1}(0)\cap  (K_i\times K_i)\setminus \Delta_i$ and all $\hat u \in {\rm CMI}(M,\R^n)$ with $|\hat u-u|<\epsilon$ on $K_{i+1}$. For the density, given $u\in {\rm CMI}(M,\R^n)$, a Runge compact set $K\subset M$ with $K_i\subset K$, and a number $\epsilon>0$, an inspection of the proof of \cite[Theorem 3.4.1]{AlarconForstnericLopez2021Book} provides a conformal minimal immersion $\hat u\colon K\to \r^n$ such that $|\hat u-u|<\epsilon$ on $K$ and $d(\sigma_{\hat u})_{(p,q)}$ is surjective for all  $(p,q)\in \sigma_{\hat u}^{-1}(0)\cap  (K_i\times K_i)\setminus \Delta_i$.
By  Theorem \ref{th:Runge} we may assume that    $\hat u\in\CMI(M,\R^n)$, hence $\Theta_i$ is dense.
\end{proof}
\begin{proof}[Proof of (ix)]   Let $E=\{bK_j\colon j\in \n\}$  (see \eqref{eq:exhaustion}) and set 
\[
\Fcal(u,C)=\inf_C |u|  \in [0,+\infty) \quad \text{for every $u\in  {\rm CMI}(M,\R^n)$ and $C\in E$.}
\]
The function  $\Fcal(\cdot,C): {\rm CMI}(M,\R^n)\to [0,+\infty)$ is  continuous for each $C\in E$.  Let $i\in \n$ and fix $u$, $K$, and $\epsilon$ as in Claim \ref{cl:patron} {\rm (B)}.   Let $j\in \N$ be so large that $K\subset \mathring K_j$ and choose a smoothly bounded compact domain $K'\subset M$ such that $K\subset K'\subset\mathring K_j$ and $K'$ is a strong deformation retract of $K_j$. Up to post-composing $u$ with a small translation in $\R^n$, assume without loss of generality that $|u|>0$ everywhere on $bK'$. By \cite[Lemma 3.11.1]{AlarconForstnericLopez2021Book} (see also \cite[Lemma 7.2]{AlarconForstnericLopez2016MZ}), there is a conformal minimal immersion $\tilde u :K_j\to\R^n$ such that $\|\tilde u-u\|_{K'}<\epsilon$ and $\inf_{bK_j}|\tilde u|>i$. Further, by Theorem \ref{th:Runge} we may assume that $\tilde u \in {\rm CMI}(M,\R^n)$, and hence  $\tilde u$ lies in  the set $\Lambda_i$ given in \eqref{eq:Lambda-i}. Thus, $\Lambda_i$  is dense in $ {\rm CMI}(M,\R^n)$ and Claim \ref{cl:patron} ensures that 
\[
\bigcap _{i\in \n} \Lambda_i=\big\{u\in {\rm CMI}(M,\R^n)\colon \forall i\in \n \;\exists j\in \n\; \text{such that}\; \inf_{bK_j} |u| >i\big\}
\]
is a dense $G_\delta$ subset in $ {\rm CMI}(M,\R^n)$. To complete the proof, note that every immersion  $u\in \bigcap _{i\in \n} \Lambda_i$ is almost proper. Indeed, for any such $u$ there is a strictly increasing sequence $\{n_k\}_{k\in\N}\subset\N$ such that $\lim_{k\to+\infty} \inf_{bK_{n_k}}|u|=+\infty$. Thus, given a number $r>0$, we have that $u^{-1}(\{x\in\r^n\colon |x|\le r\})\cap bK_{n_k}=\varnothing$ for every large enough $k\in \n$. Since $K_{n_1}\Subset K_{n_2}\Subset\cdots\subset \bigcup_{k\in\N} K_{n_k}=M$ is an exhaustion of $M$ by compact domains, it turns out that every connected component of $u^{-1}(\{x\in\r^n\colon |x|\le r\})$ is compact, proving that $u:M\to\R^n$ is an almost proper map and completing the proof. 
\end{proof}
\section{Proof of the corollaries and Theorem \ref{th:index}, and further remarks}
\begin{proof}[Proof of Corollary \ref{co:open}]
Statements {\em (i)}, {\em (iii)}, and {\em (vii)} follow straightforwardly from Theorem \ref{th:open} {\em (i)}, {\em (iii)}, and {\em (vii)}, the Baire category theorem, and the fact that a countable intersection of $G_\delta$ sets is again a $G_\delta$. Indeed, for {\em (iii)} use that a subset $W\subset \r^n$ is dense if and only if ${\rm dist}(W,\{q\})=0$ for all $q\in \Q^n$, and  for {\em (vii)} that every meagre set in $A_j\subset \r^n$ is contained in a countable union of compact sets with empty interior.

To prove {\em (ii)} recall that the space  $ \CMI(\overline\d,\R^n)$ of all conformal minimal immersions $\c\supset \overline \d\to\r^n$, endowed with the compact-open topology, is separable by  Claim \ref{cl:disc-sep}.
Let $\Dcal=\{f_i\in \CMI(\overline\d,\R^n)\colon i\in\N\}$ be a dense countable subset of $\CMI(\overline\d,\R^n)$, and note that the set  $\Sigma=\{u\in \CMI(M,\R^n)\colon \text{$u$ rebuilds $f_i$ for all $i\in \N$}\}$ equals  the set of immersions in $\CMI(M,\R^n)$ rebuilding every conformal minimal disc, while Theorem \ref{th:open}{\em (ii)} shows that   $\Sigma$ is a dense $G_\delta$ set in $\CMI(M,\R^n)$.

Let us check {\em (iv)}; the proofs of  {\em (v)} and  {\em (vi)} follow analogously and we leave the details out.  Let $\{B_j\subset \r^n\colon j\in \n\}$ be a sequence of open balls being a countable basis of the Euclidean topology in $\r^n$. For each $j\in \n$ let $\rho_j:\r^n\to \r$ be a weight function on $\r^n$ with the support ${\rm supp}(\rho_j)\Subset B_j$, $0\leq \rho_j\leq 1$, and $\rho_j|_{B_j'}=1$ for some ball $B_j'\Subset B_j$. By Theorem \ref{th:open}{\em (iv)} and the Baire category theorem, the set $X$ of all immersions  $u\in\CMI(M,\R^n)$ such that  $\sup_M |K_u| (\rho_j\circ u)=+\infty$ for all $j\in \n$ is a dense $G_\delta$ subset. To finish, let us show that $X$ equals the set $Y$ of all immersions $u\in {\rm CMI}(M,\R^n)$   having dense image and unbounded curvature on $u^{-1}(\Omega)$ for every open set $\Omega\subset \r^n$. Indeed, if $u\in X$ and  $\Omega\subset \r^n$ is open, there is $j\in \n$ such that $B_j\subset \Omega$. Since   $\sup_M |K_u| (\rho_j\circ u)=+\infty$ and   ${\rm supp}(\rho_j)\subset B_j$, we have that $\varnothing\neq u^{-1}(B_j)\subset u^{-1}(\Omega)$. This implies that $u$ has dense image. Furthermore, since $0\leq \rho_j\leq 1$ it turns out that
\[
\sup_{u^{-1}(\Omega)} |K_u|\geq\sup_{u^{-1}(\Omega)} |K_u|(\rho_j\circ u)\geq \sup_{u^{-1}(B_j)} |K_u|(\rho_j\circ u)=  \sup_M |K_u| (\rho_j\circ u)=+\infty,
\]
which implies that  $u\in Y$. On the other hand, if $u\in Y$ and $j\in \n$ then $u^{-1}(B_j')\neq \varnothing$, and since $\rho_j|_{B_j'}=1$ we infer that
\[      
 \sup_M |K_u| (\rho_j\circ u)=\sup_{u^{-1}(B_j)} |K_u|(\rho_j\circ u)\geq \sup_{u^{-1}(B_j')} |K_u| (\rho_j\circ u)=\sup_{u^{-1}(B_j')} |K_u|=+\infty,
\]
where the last equality holds since $u\in Y$ and $B_j'$ is open in $\r^n$. This proves  that $u\in X$ and concludes the proof.
\end{proof}
 \begin{proof}[Proof of Theorem \ref{th:index}]
 Let $M$ be  an open Riemann surface and $B\subset \r^3$ be an open ball. Reasoning as in the proof of Corollary \ref{co:open} {\em (iv)}, it suffices to show that  the set $X$ of all immersions $u\in {\rm CMI}(M,\R^3)$ with $u^{-1}(B)\neq \varnothing$  and having infinite index of stability on $u^{-1}(B)$   is a dense $G_\delta$ subset. For this, take an exhaustion $K_j$, $j\in \n$, of $M$ as in \eqref{eq:exhaustion}, and for each $i\in \n$ set 
 \[
 \Lambda_i=\{u\in   {\rm CMI}(M,\R^3)\colon \sup_{j\in \n}  \Ical_B(u,j)>i\},
 \]
 where $ \Ical_B(u,j)$ is the index of stability of $u$ on $\mathring K_j\cap u^{-1}(B)$, that is to say, the supremum of the stability index of $u|_K$ among all the smoothly bounded compact  domains $K\subset \mathring  K_j\cap u^{-1}(B)$. Since the function $\Ical_B(\cdot ,j):  {\rm CMI}(M,\R^3)\to \r$ takes its values in $\z_+$ and is lower semicontinuous by Cauchy estimates, we have that $\Lambda_i$ is open for all $i\in \n$. Observe there are conformal minimal discs $\d\to B$ with index of stability as big as desired (e.g.,  pieces of suitable helicoids with axes intersecting $B$), and hence an analogous argument to that in the proof of Theorem \ref{th:open}{\em (iii)--(vi)} shows that $\Lambda_i$ is dense for all $i\in \n$.  Since $X=\bigcap_{i\in \n}\Lambda_i$, the Baire category theorem completes the proof in view of Claim \ref{cl:metrizable}.
\end{proof} 
\begin{proof}[Proof of Corollary \ref{co:DP}]
Item {\em (i)} follows trivially from Theorem \ref{th:open} {\em (viii)}, since a conformal minimal immersion $M\to \r^n$, $n\geq 5$, self-intersects nicely if and only if it is injective. 
The first part of {\em (ii)} follows from \cite[Theorem 3.4.1(c)]{AlarconForstnericLopez2021Book} and the ideas in the proof of Theorem \ref{th:open} {\em (viii)}. Note that if $u\in {\rm CMI}(M,\R^4)$ has simple double points and $K\subset M$ is compact then the set $\{p\in K\colon u^{-1}(u(p))\cap K\neq \varnothing\}$  is finite, hence ${\rm DP}(u)$ is countable. Likewise, if $u\in {\rm CMI}(M,\R^3)$ self-intersects nicely then the set $\{p\in K\colon u^{-1}(u(p))\cap K\neq \varnothing\}$ is a countable union of regular curves and boundary points. 
The proof of items {\em (ii)} and {\em (iii)} is then completed by Corollary \ref{co:open} {\em (ii)} and the following claim; see Definition \ref{de:def}.
\begin{claim}\label{cl:DP}
If $M$ is an open Riemann surface and $n\in\{3,4\}$, then every immersion in $\CMI(M,\R^n)$ rebuilding every conformal minimal disc is densely non-injective.
\end{claim} 
To prove the claim, assume that $u\in\CMI(M,\R^n)$ rebuilds every conformal minimal disc and choose a smoothly bounded, relatively compact, open disc $V\subset M$. We need to prove that $V\cap{\rm DP}(u)\neq\varnothing$; see \eqref{eq:DP}. So, pick $p\in V$ and let $f:\overline\D\to\R^n$ be a conformal minimal immersion such that $f(\overline\D)$ is a planar round disc in an affine $2$-plane in $\R^n$, 
\begin{equation}\label{eq:f0up}
	f(0)=u(p),
\end{equation} 
and $du_p(T_pM)+df_0(T_0\d)=\R^n$; recall that $n\le 4$. Up to replacing $V$ by a smaller neighborhood of $p$, we can assume that
\begin{equation}\label{eq:duTpM}
du_q(T_qM)+df_z(T_z\d)=\R^n\quad \text{for all }q\in V\text{ and }z\in\D; 
\end{equation}
take into account that $df_z(T_z\d)\equiv df_0(T_0\d)$ for all $z\in\D$. Since $u$ approaches every conformal minimal disc, it turns out that for any $\epsilon>0$ there are a smoothly bounded closed disc $D$ in $M$ and a biholomorphism $\varphi:\overline\D\to D$ such that $\|u\circ\varphi-f\|_{\overline\D}<\epsilon$, and thus $\|u-f\circ\varphi^{-1}\|_D<\epsilon$ as well.  
 Since $n\in\{3,4\}$, conditions \eqref{eq:f0up} and \eqref{eq:duTpM} ensure that $u(V)\cap u(D)\neq\varnothing$ and $V\cap D=\varnothing$ whenever that $\epsilon>0$ is sufficiently small, and hence $V\cap {\rm DP}(u)\neq\varnothing$. This shows that ${\rm DP}(u)$ is dense in $M$.
\end{proof}


\begin{remark}\label{re:apen}
Let $M$ be an open Riemann surface and $n\geq 3$ be an integer.

(A) Let $\pgot\colon H_1(M,\z)\to \r^n$ be a group homomorphism and  $\CMI_\pgot(M,\R^n)\subset  \CMI(M,\R^n)$ be the subspace of conformal minimal immersions with the flux map $\pgot$ (see  \cite[Definition 2.3.2]{AlarconForstnericLopez2021Book}). A very minor modification of the proofs shows that all the results in this paper also hold  with $\CMI_\pgot(M,\R^n)$ in place of $\CMI(M,\R^n)$.

(B) Let $\aleph= (\Lambda,v,r)$ where $\Lambda\subset M$ is a closed discrete subset, $v$ is a conformal minimal immersion from a neighborhood of $\Lambda$ into $\R^n$, and $r:\Lambda\to\Z_+=\{0,1,2,\ldots\}$ is a map, and denote by ${\rm CMI}_\aleph(M,\R^n)$ the subspace of ${\rm CMI}(M,\R^n)$ consisting of those immersions agreeing with $v$ to order at least $r(p)$ at every point $p\in\Lambda$.
A very minor modification of the proofs shows that all the results in this paper also hold  with $\CMI_\aleph(M,\R^n)$ in place of $\CMI(M,\R^n)$.  In the cases of Theorem \ref{th:open}{\em (vii)}, Corollary \ref{co:open}{\em (vii)}, and Corollary \ref{co:DP}, the prescription of values $\aleph$ must be compatible with the theses.

(C) If $\pgot$ and $\aleph$ are as in (A) and (B) respectively, then the analogous results hold for the subspace $\CMI_\pgot(M,\R^n)\cap \CMI_\aleph(M,\R^n)$.

(D) Let $R$ be a compact bordered Riemann surface, and assume that $R$  {\em  appears recurrently} in $M$, meaning that  for any Runge compact set $K\subset M$ there is a smoothly bounded  compact domain $L\subset M\setminus K$ being biholomorphic to $R$ such that $K\cup L$ is Runge. Arguing as in the proofs of Theorem \ref{th:open}{\em (ii)} and Corollary \ref{co:open}{\em (ii)}, one can see that  the subspace  $\CMI_R(M,\R^n)\subset  \CMI(M,\R^n)$ of  immersions rebuilding every conformal minimal immersion $R\to \r^n$ is residual (cf.\ Definition \ref{de:def}(c)).  Likewise, if $R_j$, $j\in \n$, is a sequence of compact bordered Riemann surfaces appearing recurrently in $M$,\footnote{By  conformal surgery theory, one can construct  open Riemann surfaces $M$ such that  $R_j$ appears recurrently in $M$ for all $j\in \n$.} then 
 $\bigcap_{j\in \n} \CMI_{R_j}(M,\R^n)$ is still residual in $ \CMI(M,\R^n)$. 
\end{remark}

\begin{remark}
 The results in this paper also hold in the non-orientable framework.  Indeed, let $N$ be an open non-orientable surface furnished with a conformal structure, and denote by  $\CMI(N,\R^n)$  the space of conformal minimal immersions $N\to \r^n$ endowed with the compact-open topology. Combining the methods of proof in this paper with the results in \cite{AlarconForstnericLopez2020MAMS}, one can see that the analogues of Theorem \ref{th:open} and its corollaries, as well as those in Remark \ref{re:apen}, are valid for $\CMI(N,\R^n)$.
\end{remark}
%
%
%


\subsection*{Acknowledgements}
This research was partially supported by the State Research Agency (AEI) via the grants no.\ PID2020-117868GB-I00 and PID2023-150727NB-I00, and the ``Maria de Maeztu'' Excellence Unit IMAG, reference CEX2020-001105-M, funded by MICIU/AEI/10.13039/501100011033 and ERDF/EU, Spain.




\medskip
\noindent Antonio Alarc\'{o}n, Francisco J. L\'opez
\newline
\noindent Departamento de Geometr\'{\i}a y Topolog\'{\i}a e Instituto de Matem\'aticas (IMAG), Universidad de Granada, Campus de Fuentenueva s/n, E--18071 Granada, Spain.
\newline
\noindent  e-mail: {\tt alarcon@ugr.es}, {\tt fjlopez@ugr.es}

\end{document}